\documentclass[a4paper,12pt,twoside]{amsart}

\usepackage{amsmath,amssymb,amsfonts,amsthm,mathtools}
\usepackage[english]{babel}
\usepackage[utf8]{inputenc}
\usepackage{url}
\usepackage[pdftex,citecolor=green,linkcolor=red]{hyperref}
\usepackage[a4paper,inner=3cm,outer=3cm,top=4cm,bottom=4cm,pdftex]{geometry}
\usepackage{fancyhdr}

\usepackage{todonotes}
\newtheorem{theorem}{Theorem}[section]

\newtheorem{lemma}[theorem]{Lemma}

\newtheorem{remark}[theorem]{Remark}
\numberwithin{equation}{section}

\newcommand\eps{\varepsilon}
\newcommand\Z{\mathbb{Z}}
\newcommand\E{\mathbb{E}}
\newcommand\Prb{\mathbb{P}}

\newcommand\cD{{\mathcal D}}
\newcommand\cE{{\mathcal E}}
\newcommand\cI{{\mathcal I}}
\newcommand\cP{{\mathcal P}}
\newcommand\1{{\mathbf 1}}
\newcommand{\cond}{\operatorname{cond}}
\newcommand{\re}{\operatorname{Re}}

\renewcommand{\phi}{\varphi}

\begin{document}

\author[O. Klurman]{Oleksiy Klurman}
\address{School of Mathematics, University of Bristol, Woodland Road, Bristol, UK}
\email{lklurman@gmail.com}

\author[I. E. Shparlinski]{Igor E. Shparlinski}
\address{School of Mathematics and Statistics, University of New South Wales, Sydney, NSW 2052, Australia}
\email{igor.shparlinski@unsw.edu.au}

\author[J. Ter\"av\"ainen]{Joni Ter\"av\"ainen}
\address{Department of Pure Mathematics and Mathematical Statistics, University of Cambridge, Cambridge CB3 0WB, UK}
\email{joni.p.teravainen@gmail.com}

\title{On the generation of multiplicative groups by small primes}

\begin{abstract}
 Motivated by a question of Regev arising from his improved quantum factoring algorithm, we study how many small primes are needed to generate the group $(\Z/q\Z)^\times$ when each prime may be used with exponent only $0$ or $1$.  We prove that, for every fixed $\eps>0$ and $A>0$, there is an absolute constant $C_*$ and a set of at most $(\log Q)^{1+\eps}$ primes, all at most $(\log Q)^{C_*(A+1)}$, such that for all but $O(Q(\log Q)^{-A})$ (with the implied constant depending only on $\eps$ and $A$) integers $q\leq Q$ every element of $(\Z/q\Z)^\times$ is a product of a subset of these primes modulo $q$.  The exponent $1+\eps$ in the number of primes is best possible up to the arbitrary $\eps$ in the exponent.
\end{abstract}


\maketitle

\section{Introduction}

\subsection{Motivation and the main result}
In 1994, Shor developed a celebrated quantum algorithm for factoring integers and computing discrete logarithms in polynomial time~\cite{shor94,shor97}.  In the factoring setting, the algorithm factors an $n$-bit integer using a quantum circuit with $\widetilde O(n^2)$ gates; here and below the $\widetilde O$ notation suppresses powers of $\log n$.  A recent breakthrough of Regev~\cite{regev} introduced a multidimensional variant with $\widetilde O(n^{3/2})$ gates, conditionally on a number-theoretic generation statement about $(\Z/N\Z)^\times$ by very small primes. A recent important work of Pilatte~\cite{pilatte} proved a version of Regev's conjecture, which was sufficient to deduce unconditional correctness of a slight modification of Regev's algorithm and related variants.

One form of Pilatte's generation result~\cite{pilatte} can be summarised as follows (see~\cite[Corollary~1.4]{pilatte}).  Let $N>2$, let $d=\lceil\sqrt{\log N}\rceil$ and let
$$
X=d^{10^3d}=\exp((\log N)^{1/2+o(1)}).
$$ 
If $p_1,\ldots,p_d$ are independent and uniformly distributed random primes at most $X$ and not dividing $N$, then, with probability $1+o(1)$, every element $x$ of the subgroup $\langle p_1,\ldots,p_d\rangle\leq (\Z/N\Z)^\times$ can be written as
$$
x\equiv \prod_{i=1}^d p_i^{e_i}\bmod N,\qquad \textnormal{for some integers}\quad |e_i|\leq e^{O(d)}.
$$ 
This gives the existence of short signed exponent vectors in a randomly generated subgroup.  In this paper we consider a Boolean exponent version in which the exponents are restricted to $0$ and $1$, and we ask for a single set of primes, depending only on $Q$, whose subset products cover the whole group $(\Z/q\Z)^\times$ for almost all moduli $q\leq Q$.

As usual, we write $O_{\varrho}(\cdot)$ to indicate that the implied constant may depend on the parameter, or
a vector of parameters, $\varrho$; we apply the same convention to its equivalents $\ll_{\varrho}$ and 
$\gg_{\varrho}$ (and absence of parameters indicates absolute constants), see also Section~\ref{sec:not}. 

Given a set $\cP$ of primes, the problem is whether every unit modulo $q$ can be represented in the form
$$
\prod_{p\in\cP} p^{a_p}\bmod q, \qquad a_p\in\{0,1\}.
$$
A cardinality consideration directly implies that $|\cP|\geq\log_2(\varphi(q))\gg \log q$ is necessary for all moduli $q\geq 3$. Thus, we must have  $|\cP|\gg \log Q$ if the Boolean products of elements of $\cP$ represent every unit for at least $Q^{1/10}$ moduli $q\leq Q$, say.  A strong conditional result can be obtained by assuming the Generalised Riemann Hypothesis (GRH).  Fix $B>4$ and set $P=(\log q)^B$ with $q$ large.  A standard consequence of GRH gives, uniformly for every non-principal character $\chi\bmod q$,
$$
\left|\sum_{p\leq P}\chi(p)\right|\leq \frac{1}{100}\pi(P)
$$
for all sufficiently large $q$; see~\cite[Equation~(13.21)]{Mont-Vaughan}. The primes dividing $q$ do not contribute to this sum and $\omega(q)\ll \log q\ll \pi(P)^{1/2}$, say, so the contribution of such primes is negligible. A simple orthogonality argument then shows that the subset products of the primes $p\leq P$ with $p\nmid q$ represent every element of $(\Z/q\Z)^\times$.  Theorem~\ref{thm_generation} is an unconditional almost-all approximation to this GRH result.  It uses at most $(\log Q)^{1+\eps}$ primes, the optimal power of $\log Q$ up to the arbitrary $\eps$, and saves an arbitrary fixed power of $\log Q$ in the exceptional set.

\begin{theorem}[Generation of $(\Z/q\Z)^\times$ by subset products of small primes]\label{thm_generation}
There is an absolute constant $C_*\geq 1$ with the following property.  Let $\eps\in(0,1)$ and $A>0$.  For every $Q\geq 3$, there exists a set $\cP\subseteq[1,(\log Q)^{C_*(A+1)}]$ of primes, with
$$
|\cP|\leq (\log Q)^{1+\eps},
$$
such that for all but
$$
 O_{\eps,A}\bigl(Q(\log Q)^{-A}\bigr)
$$
positive integers $q\leq Q$ we have
\begin{align}\label{eq:Pset}
\left\{\prod_{p\in\cP}p^{a_p}\bmod q\colon a_p\in\{0,1\}\right\}= (\Z/q\Z)^\times .
\end{align}
\end{theorem}

In particular, taking $A=1$ gives an exceptional set $O_\eps(Q/\log Q)$ while keeping the primes below a fixed power $(\log Q)^{2C_{*}}$, independent of $\eps$, which could be calculated from the proof.

In the proof, we include among the exceptional moduli those divisible by at least one of the chosen primes.  This harmless step is necessary to have equality as opposed to containment, since a subset product involving a prime divisor of $q$ is not a unit modulo $q$.

\subsection{Optimality of the result.} The exceptional set in Theorem~\ref{thm_generation} is unavoidable with present technology, even if one requires the set in~\eqref{eq:Pset} cover $(\mathbb{Z}/q\mathbb{Z})^{\times}$ instead of being equal to it. In fact, one cannot obtain a better than super-polylogarithmic saving for the size of the exceptional set without settling Vinogradov's conjecture on the least quadratic nonresidue. The negation of Vinogradov's conjecture states that there exists $\eta>0$ such that there are infinitely many primes $p$ for which all the integers in $[1,p^{\eta}]$ are quadratic residues modulo $p$. Supposing that this holds for some $\eta\in (0,1)$, and fixing some $A$ in the theorem statement, we can choose arbitrarily large values of $Q\geq 3$ so that there exists a prime $q_0\in [(\log Q)^{\eta^{-1}C_*(A+1)}, 2(\log Q)^{\eta^{-1}C_*(A+1)}]$ for which all the primes up to $q_0^{\eta}$ are quadratic residues modulo $q_0$. Now, for every positive integer $q$ that is a multiple of $q_0$, there exists at least one unit $a$ modulo $q$ whose reduction modulo $q_0$ is a quadratic nonresidue. Hence, this unit $a\bmod q$ cannot be represented as a product of primes from $[1, (\log Q)^{C_*(A+1)}]\subseteq  [1,q_0^{\eta}]$ (even if one uses any nonnegative exponents). Hence, there are at least $c Q/(\log Q)^{\eta^{-1}C_*(A+1)}$ exceptional positive integers $q\leq Q$ with some constant $c > 0$.

\subsection{Proof sketch}

Fix $\eps\in(0,1)$ and $A>0$.  We set
\begin{equation}\label{eq: Y P I}
Y=(\log Q)^K, \qquad P=(\log Q)^C,
\qquad \cI=\{p\leq P\colon p\text{ prime}\},
\end{equation}
where $c_1(A+1) \le K \le c_2(A+1)$, for some absolute constants $c_2\ge c_1 > 0$ and $C=L_0K\ll A+1$, with $L_0$ an admissible exponent in Linnik's theorem.  Each prime in $\cI$ is selected independently with probability
$$
\frac{0.99(\log Q)^{1+\eps}}{|\cI|};
$$
the resulting random set is denoted by $\cP_Q$.  A Chernoff bound gives $|\cP_Q|\leq (\log Q)^{1+\eps}$ with probability $1+o(1)$. 
 We also record the high probability estimate
$$
\sum_{p\in\cP_Q}\frac{1}{p}\leq (\log Q)^{-A},
$$
which ensures that excluding moduli divisible by one of the chosen primes incurs an exceptional set of size only $ O_{\eps,A}(Q(\log Q)^{-A})$.

The main number-theoretic input is Lemma~\ref{lem:ambient}.  It says that, outside a set of $O_{\eps,A}(Q(\log Q)^{-A})$ moduli $q\leq Q$, every primitive character $\chi^*$ inducing a non-principal character modulo such a $q$ satisfies the separation bound
\begin{equation}\label{eq:intro-separation}
D_{\cI}(\chi^*)
\coloneqq \sum_{p\leq P}(1-\re\chi^*(p))
\geq (\log Q)^{-\eps/2}|\cI|.
\end{equation}
For conductors $d\leq Y$, this follows from a quantitative form of Linnik's theorem: for every $\eta>0$ and every reduced residue class $a\bmod d$, one has
$$
\pi(x;d,a)\gg_\eta \frac{x}{d^{1+\eta}\log x}\qquad \text{for } x\geq d^{L_0},
$$
with an absolute Linnik exponent $L_0$.  Summing this lower bound over all reduced residue classes $a\bmod d$, weighted by $1-\re\chi^*(a)$, gives $D_{\cI}(\chi^*)\gg_\eta |\cI|d^{-\eta}$ after slightly changing $\eta$.  With $d\leq Y=(\log Q)^K$ and $\eta$ sufficiently small in terms of $\eps/K$, this implies~\eqref{eq:intro-separation}.

For conductors $d>Y$, we use the grand density theorem, as stated in~\cite[Theorem~10.4]{IwKow}.  If $L(s,\chi^*)$ is zero-free in
$$
\textnormal{Re}(s)>1-\sigma, \qquad |\textnormal{Im}(s)|\leq P^\sigma,
$$
then the explicit formula and partial summation give, for some constant $B_0>0$, the bound
$$
\left|\sum_{p\leq P}\chi^*(p)\right|
 \ll P^{1-\sigma}(\log QP)^{B_0}=o(|\cI|),
$$
provided $C\sigma$ is sufficiently large.  Hence $D_{\cI}(\chi^*)\geq |\cI|/2$.  The remaining exceptional moduli are those divisible by a conductor $d>Y$ for which some primitive character modulo $d$ has a zero in the above rectangle.  The grand density theorem shows that the number of such moduli $2 \le q\leq Q$ is at most
$$
QY^{-1+O(\sigma)}P^{O(\sigma^2)}(\log QP)^{O(1)},
$$
which is $O_{\eps,A}(Q(\log Q)^{-A})$ by choosing $\sigma$ fixed and sufficiently small and then choosing $K,C$ in~\eqref{eq: Y P I} suitably large. 

Finally, we transfer~\eqref{eq:intro-separation} to the random subset.  For a fixed pair $(q,\chi)$, with $q$ outside the exceptional set, write $\chi^*$ for the primitive character inducing $\chi$ and let $B_p$ be the indicator of the event $p\in\cP_Q$.  Then
$$
X_{q,\chi}=\sum_{p\leq P}(1-\re\chi^*(p))B_p
$$
is a sum of independent random variables in $[0,2]$, and~\eqref{eq:intro-separation} gives
$$
\E X_{q,\chi}\geq 0.99(\log Q)^{1+\eps/2}.
$$
A lower tail estimate, followed by a union bound over the $O(Q^2)$ pairs $(q,\chi)$, shows that with high probability we have
\begin{equation}\label{eq:intro-random}
D_{\cP_Q}(\chi)
=\sum_{p\in\cP_Q}(1-\re\chi(p))
\geq 20\log Q
\end{equation}
for every non-principal character $\chi\bmod q$, after also excluding moduli divisible by any selected prime.  Character orthogonality then gives, for $a\in(\Z/q\Z)^\times$,
$$
T_q(a)\coloneqq \sum_{(b_p)\in\{0,1\}^{\cP_Q}}
\1_{\prod_{p\in\cP_Q}p^{b_p}\equiv a\bmod q}=\frac{1}{\phi(q)}\sum_{\chi\bmod q}\overline\chi(a)
\prod_{p\in\cP_Q}(1+\chi(p)).
$$
The principal character contributes $2^{|\cP_Q|}/\phi(q)\geq 2^{|\cP_Q|}/Q$.  
For every non-principal $\chi$, we see that~\eqref{eq:intro-random} and the arithmetic-geometric mean inequality give
$$
\left|\prod_{p\in\cP_Q}(1+\chi(p))\right|
\leq 2^{|\cP_Q|}\exp\left(-\frac{1}{4}D_{\cP_Q}(\chi)\right)
\leq 2^{|\cP_Q|}Q^{-5}.
$$
Thus the total non-principal contribution is smaller than the principal contribution, so $T_q(a)>0$ for every unit $a\bmod q$.

\subsection{Notation and conventions}
\label{sec:not}
From this point on, the constants implicit in the symbols $O$, $o$, $\ll$ and $\gg$ may depend on the fixed parameters $\eps$, $\eta$ and $A$, 
so we do not indicate this in subscripts anymore. 
%%other parameters are on which they may depend are indicated by in the subscripts.

For a finite set $S$, we write $|S|$ for its cardinality.

The letter $p$, with or without indices, always denotes a prime number. 

Let $\chi_0$ denote the principal character.  We write $\cond\chi$ for the conductor of a Dirichlet character $\chi$.  As usual, $\phi(q)$ is Euler's totient function.  For a real number $x\geq 2$, $\pi(x)$ denotes the number of primes up to $x$, and $\pi(x;q,a)$ denotes the number of primes $p\leq x$ with $p\equiv a\bmod q$.

\subsection{Acknowledgments}
The authors would like to thank Oded Regev for attracting their attention to this problem. 

During the preparation of this work
I.S.\ was  supported by the Australian Research Council Grants  DP230100530 and DP230100534 and by a
Knut and Alice Wallenberg Fellowship. 
J.T.\ was supported by funding from the European Union's Horizon Europe research and innovation programme under ERC grant agreement No.~101162746 and Marie Sk\l{}odowska-Curie grant agreement No.~101058904.

This work started while the authors were visiting   Institut Mittag-Leffler, Sweden,
during the programme `Analytic Number Theory' in January--April of 2024,
whose hospitality and support are gratefully acknowledged.

\section{Prime sums and separation bounds}\label{sec:prime-sums}

In this section we establish separation bounds, that is, lower bounds on the sum $\sum_{p\leq P}(1-\textnormal{Re}(\chi(p)))$ where $\chi$ is a character modulo $q$. For characters of small conductor, the input needed is a quantitative form of Linnik's theorem. 

\begin{lemma}[Quantitative Linnik theorem]\label{lem:linnik}
There is an absolute constant $L_0\geq 1$ such that, for every $\eta>0$, there is a constant $c_\eta>0$ for which
\begin{equation}\label{eq:linnik}
\pi(x;d,a)\geq c_\eta\frac{x}{d^{1+\eta}\log x}
\end{equation}
whenever $\gcd(a,d)=1$ and $x\geq d^{L_0}$.
\end{lemma}

\begin{proof}
This follows from~\cite[Theorem~18.7]{IwKow} (where we use $\Lambda(n)\leq \log x$ for $n\leq x$ and bound the contribution of higher prime powers trivially), combined with Siegel's bound $1-\beta\gg  d^{-\eta}$ for possible real zeros of Dirichlet $L$-functions modulo $d$.
\end{proof}

\begin{remark}
The constant $c_{\eta}$ in Linnik's theorem is ineffective, and hence the exceptional bound in Theorem~\ref{thm_generation} is also ineffective. However, for $\eta=1/2$ the constant $c_{\eta}$ in Linnik's theorem is effective, and a slight modification of the arguments (keeping the contribution of the exceptional zero in Linnik's theorem) would allow one to prove Theorem~\ref{thm_generation} using only this value of $\eta$, hence making the result effective. We leave the details to the interested reader.
\end{remark} 

\begin{lemma}[Small conductor separation]\label{lem:linnik-sep}
For every $\eta>0$ there is a constant $\widetilde c_\eta>0$ such that the following holds.  Let $\chi$ be a primitive non-principal character of conductor $d$, and let $x\geq d^{L_0}$.  Then
\begin{equation}\label{eq:linnik-sep}
\sum_{p\leq x}(1-\re\chi(p))\geq \widetilde  c_\eta\,\pi(x)d^{-\eta}.
\end{equation} 
\end{lemma}

\begin{proof}
Apply Lemma~\ref{lem:linnik} with parameter $\eta/2$.  Since $\chi$ is non-principal,
$$
\sum_{a\in(\Z/d\Z)^\times}\chi(a)=0,
\qquad
\sum_{a\in(\Z/d\Z)^\times}(1-\re\chi(a))=\phi(d).
$$
The primes dividing $d$ give non-negative contributions to the left-hand side of~\eqref{eq:linnik-sep}, and therefore
$$
\begin{aligned}
\sum_{p\leq x}(1-\re\chi(p))
& \geq \sum_{a\in(\Z/d\Z)^\times}(1-\re\chi(a))\pi(x;d,a) \\
& \geq c_{\eta/2}\frac{x}{d^{1+\eta/2}\log x}\phi(d).
\end{aligned}
$$
Using the elementary bound $\phi(d)\gg d/\log\log (d+2)$, see~\cite[Theorem~2.9]{Mont-Vaughan}, 
and the prime number theorem in the very crude form $\pi(x)\ll  x/\log x$, we obtain~\eqref{eq:linnik-sep}.
\end{proof}

The input we need for separation bounds for characters of large conductor  is Lemma~\ref{lem:density} below, 
which gives us a strong bound on prime character sums outside a small exceptional set of moduli. 
To establish this, we first need a result giving a strong character sum bound assuming a suitable zero-free region. 
This result is certainly well-known, and is give, for example, by Pilatte~\cite[Proposition~2.9]{pilatte}; we give a different 
proof for the sake of completeness.

\begin{lemma}[Character sum bound from a zero-free rectangle]\label{lem:explicit-formula}
Let $0<\sigma\leq 1/7$, let $P\geq 3$, and let $\chi$ be a primitive non-principal character of conductor at most $Q$.  Suppose that $L(s,\chi)$ has no zeros in the rectangle
\begin{equation}\label{eq:zero-free-rectangle}
\textnormal{Re}(s)\geq 1-\sigma,
\qquad |\textnormal{Im}(s)|\leq P^\sigma.
\end{equation}
Then
\[
\left|\sum_{p\leq P}\chi(p)\right|
\ll P^{1-\sigma}(\log QP)^3. 
\] 
\end{lemma}

\begin{proof}
Set $T=P^\sigma/2\geq 1/2$ and 
$$
\psi(x,\chi)=\sum_{n\leq x}\chi(n) \Lambda(n),\quad \vartheta(x,\chi)=\sum_{p\leq x}\chi(p) \log p.
$$
By Perron's formula, uniformly for $2\leq x\leq P$ we have
$$
\psi(x,\chi)=-\frac{1}{2\pi i}\int_{1+1/\log x-iT}^{1+1/\log x+iT}\frac{L'(s,\chi)}{L(s,\chi)}\frac{x^s}{s}\,\textnormal{d}s+O\left(\frac{x(\log QP)^2}{T}+\log(QP)\right).
$$
We shift the line of integration to $\textnormal{Re}(s)=1-\sigma+1/\log(QP)$. Since there are no poles of $L'(s,\chi)/L(s,\chi)$ in the rectangle $\textnormal{Re}(s)\geq  1-\sigma$, $|\textnormal{Im}(s)|\leq 2T$, a standard pointwise bound (see~\cite[Lemma 11.1]{Mont-Vaughan}) gives $|L'(s,\chi)/L(s,\chi)|\ll (\log QP)^2$ on the boundary of the rectangle with vertices $1+1/\log x\pm iT, 1-\sigma+1/\log(PQ)\pm iT$. Hence, by the residue theorem and the fact that 
$$
\int_{-T}^{T}\frac{1}{\left|1-\sigma+1/\log(QP)+iy\right|}x^{1-\sigma}\,\textnormal{d}y\ll x^{1-\sigma}\log(3T),\quad \int_{1-\sigma+1/\log(QP)}^{1+1/\log x} \frac{x^{u}}{T}\,\textnormal{du}\ll \frac{x}{T}
$$
we obtain for $2\leq x\leq P$ the bound
$$
|\psi(x,\chi)|\ll P^{1-\sigma}(\log QP)^3.
$$
Removing prime powers from $\psi(x,\chi)$ changes the sum by
$O(P^{1/2}(\log P)^2)$,
so
\begin{equation}\label{eq:theta-bound}
|\vartheta(x,\chi)|\ll P^{1-\sigma}(\log QP)^3    
\end{equation}
since $\sigma\leq 1/2$.  
Now, the result follows from the partial summation identity
$$
\sum_{p\leq P}\chi(p)=\frac{\vartheta(P,\chi)}{\log P}+\int_{2}^{P}\frac{\vartheta(u,\chi)}{u(\log u)^2} \,\textnormal{d}u
$$
combined with~\eqref{eq:theta-bound}. 
\end{proof}

\begin{lemma}[Grand density exceptional set]\label{lem:density}
There is an absolute constant $B_0\geq 3$ such that the following holds.  Let $Q\geq 3$, $P\geq 3$, $1\leq Y\leq Q$, and $0<\sigma\leq 1/7$.  Then there is a set $\cE_0\subseteq[1,Q]$ with
\begin{equation}\label{eq:density-exceptional}
  |\cE_0|\ll QY^{-1+6\sigma}P^{3\sigma^2}(\log QP)^{B_0}
\end{equation}
such that, for every $q\leq Q$, $q\notin\cE_0$, and every character $\chi\bmod q$ with $\cond\chi>Y$, if $\chi^*$ is the primitive character inducing $\chi$, then
\begin{equation}\label{eq:density-prime-sum}
 \left|\sum_{p\leq P}\chi^*(p)\right|
 \ll P^{1-\sigma}(\log QP)^{B_0}.
\end{equation}
\end{lemma}

\begin{proof}
We combine Lemma~\ref{lem:explicit-formula} with the grand density theorem for Dirichlet $L$-functions; see~\cite[Theorem~10.4]{IwKow}.  If $N(\alpha,T,\chi)$ denotes the number of zeros $\rho=\beta+i\gamma$ of $L(s,\chi)$ with $\beta\geq \alpha$ and $|\gamma|\leq T$, counted with multiplicity, set
$$
N(\alpha,T,D)=\sum_{d\leq D}\,\,\sideset{}{^*}\sum_{\chi\bmod d}N(\alpha,T,\chi),
$$
where $\Sigma^*$ denotes summation over primitive characters.  The grand density theorem gives, for $0<\sigma\leq 1/2$ and $T\geq 1$,
\begin{equation}\label{eq:grand-density}
N(1-\sigma,T,D)
\ll (D^2T)^{3\sigma}(\log DT)^{B_0},
\end{equation}
with a suitable absolute constant $B_0\geq 3$.

Set $T=P^\sigma$.  For a dyadic interval $(Z,2Z]$, let $\cD(Z)$ be the set of conductors $d\in(Z,2Z]$ for which some primitive character modulo $d$ has a zero in the rectangle~\eqref{eq:zero-free-rectangle}.  
By~\eqref{eq:grand-density},
$$
|\cD(Z)|\ll (Z^2P^\sigma)^{3\sigma}(\log ZP)^{B_0}.
$$
We place in $\cE_0$ all moduli $q\leq Q$ divisible by at least one $d\in\cD(Z)$ for some dyadic $Z\geq Y/2$.  Since $-1+6\sigma<0$, the number of such moduli is
$$
\begin{aligned}
|\cE_0|
& \ll \sum_{\substack{Z\in[Y/2,Q]\\ Z=2^j}}
\sum_{d\in\cD(Z)}\frac{Q}{d}  \\
& \ll QP^{3\sigma^2}(\log QP)^{B_0}
\sum_{\substack{Z\in[Y/2,Q]\\ Z=2^j}} Z^{-1+6\sigma}  \\
& \ll QY^{-1+6\sigma}P^{3\sigma^2}(\log QP)^{B_0}.
\end{aligned}
$$
If $q\notin\cE_0$ and $\chi\bmod q$ has conductor $d>Y$, then $d$ is not in any of the sets $\cD(Z)$, so $L(s,\chi^*)$ is zero-free in~\eqref{eq:zero-free-rectangle}.  Then Lemma~\ref{lem:explicit-formula} gives~\eqref{eq:density-prime-sum}.
\end{proof}

We now deduce a separation bound by combining the small conductor estimate (Lemma~\ref{lem:linnik-sep}) with Lemma~\ref{lem:density} in the large  conductor case.

\begin{lemma}[Separation bound on the full prime segment]\label{lem:ambient}
Let $\eps\in(0,1)$ and $A>0$.  There is an absolute constant $C_*\geq 1$ such that, for all sufficiently large $Q$, the following holds with some $K=K(A)\geq 1$ and some $C=C(A)$ satisfying
$$
A+3\leq C\leq C_*(A+1).
$$
Set
$$
Y=(\log Q)^K,
\qquad P=(\log Q)^C,
\qquad \cI=\{p\leq P\colon p\text{ prime}\}.
$$
Then there is a set $\cE_0\subseteq[1,Q]$ with
$$
    |\cE_0|\ll  Q(\log Q)^{-A}
$$
such that the following holds. 
 If $q\leq Q$, $q\notin\cE_0$, and $\chi\bmod q$ is non-principal, then, writing $\chi^*$ for the primitive character inducing $\chi$, one has
\begin{equation}\label{eq:ambient-separation}
    \sum_{p\in\cI}(1-\re\chi^*(p))
    \geq (\log Q)^{-\eps/2}|\cI|.
\end{equation}
\end{lemma}

\begin{proof}
Let $L_0$ be as in Lemma~\ref{lem:linnik}, and let $B_0$ be as in Lemma~\ref{lem:density}.  Choose a fixed 
$\sigma \in (0,1/7)$  so small that
\begin{equation}\label{eq:sigma-choice}
6\sigma+3L_0\sigma^2\leq \frac{1}{4}.
\end{equation} 
Next choose $K$ so large in terms of $A$ and the absolute constants that
\begin{equation}\label{eq:K-choice}
L_0K\sigma\geq B_0+2,
\qquad L_0K\geq A+3,
\qquad \frac{3}{4}K-B_0\geq A+1.
\end{equation}
We may do this with $K\ll A+1$, where the implied constant is absolute.  Finally set $C=L_0K$, increasing the absolute constant $C_*$ if necessary so that $C\leq C_*(A+1)$.

Let $q$ and $\chi$ be as in the statement, and let $d=\cond\chi^*$.  Suppose first that $d\leq Y$.  Then $P\geq d^{L_0}$.  Apply Lemma~\ref{lem:linnik-sep} with $\eta=\eps/(4K)$.  Since $d\leq Y=(\log Q)^K$, we get
$$
\sum_{p\leq P}(1-\re\chi^*(p))
\gg  |\cI| d^{-\eta}
\geq c_{\eps,A}|\cI|(\log Q)^{-\eps/4}
\geq |\cI|(\log Q)^{-\eps/2}
$$
for all sufficiently large $Q$.

It remains to consider $d>Y$.  Let $\cE_0$ be the set given by Lemma~\ref{lem:density}.  From~\eqref{eq:density-exceptional}, \eqref{eq:sigma-choice}, \eqref{eq:K-choice} and $P=(\log Q)^{L_0K}$, we have
$$
\begin{aligned}
|\cE_0|
& \ll Q(\log Q)^{-K(1-6\sigma)+3L_0K\sigma^2+B_0+o(1)} \\
& \ll Q(\log Q)^{-A}.
\end{aligned}
$$
If $q\notin\cE_0$, Lemma~\ref{lem:density} gives
$$
\left|\sum_{p\leq P}\chi^*(p)\right|
\ll P^{1-\sigma}(\log QP)^{B_0}.
$$
By~\eqref{eq:K-choice}, $C\sigma=L_0K\sigma\geq B_0+2$, and hence this is $o(P/\log P)=o(|\cI|)$.  Therefore, for all sufficiently large $Q$,
$$
\sum_{p\leq P}(1-\re\chi^*(p))
\geq |\cI|-\left|\sum_{p\leq P}\chi^*(p)\right|
\geq \frac{|\cI|}{2}
\geq |\cI|(\log Q)^{-\eps/2}.
$$
This concludes the proof.
\end{proof}

\section{Character separation for random primes}

We now transfer the deterministic separation from the full prime segment to a sparse random subset of primes. For this, we first need some elementary tail probability estimates.

\begin{lemma}[Elementary probability estimates]\label{lem:probability}
Let $n\geq 1$, and let $X_1,\ldots,X_n$ be independent random variables satisfying $0\leq X_j\leq 2$.  Set $X=\sum_{j=1}^nX_j$ and $\mu=\E X$.  Then
\begin{equation}\label{eq:lower-tail}
\Prb(X\leq \mu/2)\leq \exp(-c\mu)
\end{equation}
for some absolute constant $c>0$.  Also, if $B_1,\ldots,B_n$ are independent Bernoulli random variables with common mean $\rho$ and $S=\sum_{j=1}^nB_j$, then, for $0<\delta\leq 1$,
\begin{equation}\label{eq:chernoff}
\Prb(S\geq (1+\delta)\rho n)
\leq \exp\left(-\frac{\delta^2\rho n}{3}\right).
\end{equation}
\end{lemma}

\begin{proof}
For $0\leq u\leq 2$ and $t>0$, by the convexity of $u\mapsto e^{-tu}$ and the inequality $1-y\leq e^{-y}$ we have
$$
e^{-tu}\leq 1-\frac{1-e^{-2t}}{2}u\leq \exp\left(-\frac{1-e^{-2t}}{2}u\right).
$$
Taking $t=1/4$ and using independence, Markov's inequality gives
$$
\begin{aligned}
\Prb(X\leq \mu/2)
& =\Prb(e^{-tX}\geq e^{-t\mu/2})           \\
& \leq e^{t\mu/2}\prod_{j=1}^n\E e^{-tX_j} \\
&\leq \exp\left(-\left(\frac{1-e^{-2t}}{2}-\frac{t}{2}\right)\mu\right).
\end{aligned}
$$
The coefficient in parentheses is positive, so~\eqref{eq:lower-tail} follows.  The Bernoulli estimate follows from the standard Chernoff calculation
$$
\Prb(S\geq (1+\delta)\rho n)
\leq \left(\frac{e^\delta}{(1+\delta)^{1+\delta}}\right)^{\rho n}
\leq \exp\left(-\frac{\delta^2\rho n}{3}\right),
$$
valid for $0<\delta\leq 1$.
\end{proof}

\begin{lemma}[Separation for random sets of primes]\label{lem:randomsep}
Let $\eps\in(0,1)$ and $A>0$, let $Q$ be large, and let $K,C$ be the constants given by Lemma~\ref{lem:ambient}.  Set
$$
P=(\log Q)^C,
\qquad \cI=\{p\leq P\colon p\text{ prime}\},
\qquad M=(\log Q)^{1+\eps}.
$$
Let $\cP_Q\subseteq\cI$ be the random set obtained by choosing each prime in $\cI$ independently with probability
$$
     \rho=\frac{0.99M}{|\cI|}.
$$
Then, with probability $1+o(1)$, the following two assertions hold:
\begin{enumerate}
    \item $|\cP_Q|\leq M$ and
          \begin{equation}\label{eq:reciprocal-small}
        \sum_{p\in\cP_Q}\frac{1}{p}\leq (\log Q)^{-A};
          \end{equation}
    \item there is a set $\cE_Q\subseteq[1,Q]$, $|\cE_Q|\ll Q(\log Q)^{-A}$, such that for every $q\leq Q$ with $q\notin\cE_Q$, and every non-principal character $\chi\bmod q$, writing $\chi^*$ for the primitive character inducing $\chi$, one has
          \begin{equation}\label{eq:random-separation}
            \sum_{p\in\cP_Q}(1-\re\chi^*(p))\geq 20\log Q.
          \end{equation}
\end{enumerate}
\end{lemma}

\begin{proof}
By Lemma~\ref{lem:ambient} we have $C\geq A+3\geq 2$.  Since $\eps<1$, the probability $\rho=0.99M/|\cI|$ is less than $1$ for all sufficiently large $Q$.  The random variable $|\cP_Q|$ has mean $0.99M$.  Applying~\eqref{eq:chernoff} with $\delta=1/99$ gives $|\cP_Q|\leq M$ with probability $1+o(1)$.

Next,
$$
\E\sum_{p\in\cP_Q}\frac{1}{p}
=\rho\sum_{p\leq P}\frac{1}{p}
\ll \frac{M\log P\log\log P}{P}
\ll  (\log Q)^{-A-1/2},
$$
by the choice $C\geq A+3$, Mertens' estimate for $\sum_{p\leq P}1/p$, and the crude bound $|\mathcal{I}|\asymp P/\log P$.  Now, Markov's inequality gives~\eqref{eq:reciprocal-small} with probability $1+o(1)$.

Let $\cE_0$ be the exceptional set given by Lemma~\ref{lem:ambient}; we shall take $\cE_Q=\cE_0$.  Let $B_p$ denote the indicator of the event $p\in\cP_Q$.  Fix a pair $(q,\chi)$ with $q\leq Q$, $q\notin\cE_0$, and $\chi\bmod q$ non-principal, and let $\chi^*$ be the primitive character inducing $\chi$.  Set
$$
X_{q,\chi}=\sum_{p\in\cI}(1-\re\chi^*(p))B_p.
$$
For this fixed pair, the summands $(1-\re\chi^*(p))B_p$, $p\in\cI$, are independent and lie in $[0,2]$, and
$$
X_{q,\chi}=\sum_{p\in\cP_Q}(1-\re\chi^*(p)).
$$
By Lemma~\ref{lem:ambient},
$$
\E X_{q,\chi}
=\rho\sum_{p\in\cI}(1-\re\chi^*(p))
\geq 0.99M(\log Q)^{-\eps/2}
=0.99(\log Q)^{1+\eps/2}.
$$
For large $Q$, this expectation is at least $40\log Q$.  Lemma~\ref{lem:probability} therefore gives
$$
\Prb(X_{q,\chi}<20\log Q)
\leq \exp\bigl(-c_1(\log Q)^{1+\eps/2}\bigr)
$$
with an absolute constant $c_1>0$.  There are at most $Q^2$ pairs $(q,\chi)$ with $q\leq Q$.  Hence, by the union bound, the probability that~\eqref{eq:random-separation} fails for at least one such pair is
$$
O\left(Q^2\exp\bigl(-c_1(\log Q)^{1+\eps/2}\bigr)\right)=o(1).
$$
This proves the lemma.
\end{proof}

\section{Generation of the multiplicative group}

We are now ready to prove Theorem~\ref{thm_generation}.

\begin{proof}[Proof of Theorem~\ref{thm_generation}]
 It suffices to consider moduli $q\geq 3$. For bounded $Q$ the theorem is trivial by taking a large implicit constant in the exceptional set estimate, so assume that $Q$ is sufficiently large in terms of $\eps$ and $A$.

Choose a set $\cP_Q$ satisfying Lemma~\ref{lem:randomsep}; such a set exists because the probability there is $1+o(1)$.  Then $|\cP_Q|\leq (\log Q)^{1+\eps}$, every prime in $\cP_Q$ is at most $(\log Q)^{C_*(A+1)}$, and
\begin{equation}\label{eq:chosen-recip}
    \sum_{p\in\cP_Q}\frac{1}{p}\leq (\log Q)^{-A}.
\end{equation}

Let $\cE_Q$ be the exceptional set in Lemma~\ref{lem:randomsep}.  We enlarge it by also excluding moduli divisible by at least one prime in $\cP_Q$.  By~\eqref{eq:chosen-recip}, the number of additionally excluded moduli is at most
$$
\sum_{p\in\cP_Q}\frac{Q}{p}
\leq Q(\log Q)^{-A}.
$$
Hence the enlarged exceptional set still has cardinality
$$
O\bigl(Q(\log Q)^{-A}\bigr).
$$

Let $q\leq Q$ lie outside this enlarged exceptional set.  Then every $p\in\cP_Q$ is coprime to $q$.  Therefore, if $\chi\bmod q$ is induced by $\chi^*$, then $\chi(p)=\chi^*(p)$ for all $p\in\cP_Q$, and Lemma~\ref{lem:randomsep} gives
\begin{equation}\label{eq:D-final}
D_{\cP_Q}(\chi)
\coloneqq \sum_{p\in\cP_Q}(1-\re\chi(p))
\geq 20\log Q
\end{equation}
for every non-principal $\chi\bmod q$.

For $a\in(\Z/q\Z)^\times$ define
$$
T_q(a)=\sum_{(b_p)\in\{0,1\}^{\cP_Q}}
\1_{\prod_{p\in\cP_Q}p^{b_p}\equiv a\bmod q}.
$$
By orthogonality of characters on $(\Z/q\Z)^\times$,
\begin{equation}\label{eq:T-orthogonality}
\begin{split}
T_q(a)
& =\frac{1}{\phi(q)}\sum_{\chi\bmod q}\overline\chi(a)
\prod_{p\in\cP_Q}(1+\chi(p))\notag  \\
& =\frac{2^{|\cP_Q|}}{\phi(q)}
+\frac{1}{\phi(q)}\sum_{\substack{\chi\bmod q\\ \chi\ne\chi_0}}
\overline\chi(a)\prod_{p\in\cP_Q}(1+\chi(p)).
\end{split} 
\end{equation}
The principal term is at least $2^{|\cP_Q|}/Q$.  For a non-principal $\chi$, the arithmetic-geometric mean inequality and~\eqref{eq:D-final} give
\begin{align*}
\left|\prod_{p\in\cP_Q}(1+\chi(p))\right|
& \leq \left(\frac{1}{|\cP_Q|}\sum_{p\in\cP_Q}|1+\chi(p)|^2\right)^{|\cP_Q|/2} \\
& =\left(4-\frac{2D_{\cP_Q}(\chi)}{|\cP_Q|}\right)^{|\cP_Q|/2}              \\
& =2^{|\cP_Q|}\left(1-\frac{D_{\cP_Q}(\chi)}{2|\cP_Q|}\right)^{|\cP_Q|/2}   \\
& \leq 2^{|\cP_Q|}\exp\left(-\frac{D_{\cP_Q}(\chi)}{4}\right)                  \\
& \leq 2^{|\cP_Q|}Q^{-5}.
\end{align*}
Consequently the absolute value of the total non-principal contribution in~\eqref{eq:T-orthogonality} is at most $2^{|\cP_Q|}Q^{-5}$, while the principal contribution is at least $2^{|\cP_Q|}/Q$.  Thus $T_q(a)>0$ for every $a\in(\Z/q\Z)^\times$.

Every unit modulo $q$ is therefore represented as a subset product of primes in $\cP_Q$.  Since no prime in $\cP_Q$ divides $q$, every subset product is a unit modulo $q$, so the set of subset products is exactly $(\Z/q\Z)^\times$.  This proves Theorem~\ref{thm_generation}.
\end{proof}

\end{document}